\documentclass[11pt]{amsart}

\usepackage[T1]{fontenc}
\usepackage[latin1]{inputenc}
\usepackage[english]{babel}
\usepackage{amsmath,amssymb,amsthm} 
\usepackage{enumitem}
\usepackage{mymacros}

\usepackage[margin=1.5in]{geometry}
\usepackage{color}
\usepackage{hyperref}
\usepackage{bm}
\usepackage{graphicx}
\usepackage{textcomp}

\title{Cyclicity and iterated logarithms in the Dirichlet space}
\author{Alexandru Aleman}
\address{Department of Mathematics\\ Lund University}
\email{alexandru.aleman@math.lu.se}
\author{Stefan Richter}
\address{Department of Mathematics\\ The University of Tennessee\\ Knoxville, TN 37996}
\email{srichter@utk.edu}

\date{\today}
\thanks{ }

\newcommand{\la}{\langle}
\newcommand{\ra}{\rangle}

\newcommand{\MH}{\Mult(\HH)}

\numberwithin{equation}{section}

\begin{document}

\date{\today}

\bibliographystyle{plain}

\begin{abstract} Let $D(\mu)$ denote a superharmonically weighted Dirichlet space on the unit disc $\D$. We show that outer functions  $f\in D(\mu)$ are cyclic in $D(\mu)$, whenever $\log f$ belongs to the Pick-Smirnov class $N^+(D(\mu))$. If $f$ has $H^\infty$-norm less than or equal to 1, then cyclicity can also be checked via iterated logarithms. For example, we show that such outer functions $f$ are cyclic, whenever $\log(1+ \log(1/f))\in N^+(D(\mu))$. This condition can be checked by verifying that $\log(1+ \log(1/f))\in D(\mu)$.

If $f$ satisfies a mild extra condition, then the conditions also become necessary for cyclicity.
\end{abstract}
\maketitle

\section{Introduction}
Let $\D$ be the open unit disc in the complex plane $\C$, and let $\HH\subseteq \Hol(\D)$ be a reproduding kernel Hilbert space with multiplier algebra
$$\Mult(\HH)=\{\varphi: \D \to \C: \varphi f \in \HH \ \forall f\in \HH\}.$$ Multipliers define bounded linear operators on $\HH$ by $M_\varphi f= \varphi f$, and we write $$\|\varphi\|_{\MH}=\|M_\varphi\|_{\HB(\HH)}$$ for the multiplier norm of $\varphi \in \MH$.

If $\HH=H^2$ is the Hardy space of the unit disc, then $\Mult(H^2)=H^\infty$, the bounded analytic functions on $\D$. In this paper we will  be interested in the superharmonically weighted Dirichlet spaces. Let  $\mu $ be a finite nonnegative Borel measure in the closed unit disc, and write
$$U_\mu(z)= \int_{|w|<1}\log\left|\frac{1-\overline{w}z}{w-z}\right|^2 \frac{d\mu(w)}{1-|w|^2} + \int_{|w|=1} \frac{1-|z|^2}{|1-z\overline{w}|^2} d\mu(w), |z|<1.$$
The space $D(\mu)$ is defined by
$$D(\mu)=\{f\in H^2: \int_{|z|<1} |f'(z)|^2 U_\mu(z) dA(z)<\infty\},$$
and a norm on $D(\mu)$ is given by $\|f\|^2_\mu= \|f\|^2_{H^2} + \int_{|z|<1} |f'(z)|^2 U_\mu(z) \frac{dA(z)}{\pi}$. Here we have written $dA(z)=dxdy$ for 2-dimensional Lebesgue measure on $\D$.

 Thus, if $\mu=0$, then $D(\mu)=H^2$ and if $d\mu= \frac{|dz|}{2\pi}$ is normalized Lebesgue measure on $\T=\partial \D$, then $U_\mu=1$ and $D(\mu)=D$, the classical Dirichlet space. Furthermore,  the map $\mu\to U_\mu$ is a bijective correspondence between the finite nonnegative Borel measures $\mu$ on $\overline{\D}$ and the nonnegative superharmonic functions on $\D$, see  \cite{AlexandruHab}, page 75. In particular, one checks that for $0\le \alpha\le 1$  the standard weighted Dirichlet spaces $D_\alpha=\{f\in H^2: \int_{|z|<1} |f'(z)|^2 (1-|z|^2)^{1-\alpha} dA(z)<\infty\}$ are included in the collection of superharmonically weighted Dirichlet spaces. $D(\mu)$-spaces spaces have been studied by many authors. For some of their basic properties we refer the reader to \cite{RiSuMMJ}, \cite{AlexandruHab}, and \cite{ElFallahKellayMashreghiRansfordPrimer}.  It is trivial to see that all constant functions are contained in $D(\mu)$, hence $\Mult(D(\mu))\subseteq D(\mu)$.

Let $\HH$ be reproducing kernel Hilbert  space on $\D$ with reproducing kernel $k_w(z)$. $k$ is called a normalized complete Nevanlinna-Pick kernel, if there is an auxiliary Hilbert space $\HK$ and a function $b:\D\to \HK$ such that $b(0)=0$ and $k_w(z)= \frac{1}{1-\la b(z),b(w)\ra_{\HK}}.$ Spaces with normalized complete Nevanlinna-Pick kernel have been investigated, because there are many interesting examples, and because many results that are true for $H^2$ are true for such spaces. Shimorin showed in \cite{ShimorinCNP} that all $D(\mu)$-spaces have a normalized complete Nevanlinna-Pick kernel.

If $g\in \HH$, then we write $[g]$ for the closure in $\HH$ of $\{\varphi g: \varphi\in \MH\}$. A function is called cyclic for $\HH$, if $[g]=\HH$. The Pick-Smirnov class of $\HH$ is defined by
$$N^+(\HH)=\{\frac{u}{v} : u,v\in \MH, v \text{ is cyclic} \}.$$ Of course, if $\HH=H^2$, then $N^+(\HH)$ coincides with the classical Smirnov class.
In \cite{AHMRsccps} it was shown that if $k$ is a normalized complete Nevanlinna-Pick kernel, then $\HH\subseteq N^+(\HH)$. More precisely, it was shown that if $f\in \HH$ with $\|f\|\le 1$, then there
are  $\varphi, \psi\in \MH$  with $\psi(0)=0$ and $\|\psi h\|^2+\|\varphi h\|^2\le \|h\|^2$ for all $h\in \HH$ and \begin{align}\label{equ:Smirnovratio}f=\frac{\varphi}{1-\psi}.\end{align} One then shows that $1-\psi$ is cyclic, see Lemma 2.3 of \cite{AHMRsccps}. It also follows that $[f]=[\varphi]$ and many questions about cyclic functions in $\HH$ can be reduced to questions about cyclic functions in $\MH$, see Lemma 6.5 of \cite{APRSS1}.

Hence we have \begin{align}\label{equ:N^+} D(\mu)\subseteq N^+(D(\mu)).\end{align}

In the paper \cite{APRSS2} the authors gave some conditions for cyclicity of functions in radially weighted Besov spaces on the unit ball of $\C^d$ (including the Dirichlet space $D$ and the Drury-Arveson space). All of those results hold for all  spaces $D(\mu)$, even though the weight may not be radial. In the current paper we will show that if $\HH=D(\mu)$, then the validity of the results from \cite{APRSS2} can be extended to hold for a larger class of functions.

The cyclic functions in $H^2$ are the outer functions. From this and the contractive inclusion $D(\mu)\subseteq H^2$ it follows that cyclic functions in $D(\mu)$ have to be outer. Furthermore, for spaces $\HH$ with normalized complete Nevanlinna-Pick kernel it is an easy observation that $g\in \HH$ is cyclic, if and only if $1/g \in N^+(\HH)$, see e.g. \cite{APRSS2}, Theorem 3.1. Since $N^+(\HH)$ is an algebra, it follows that if $g^{-\frac{1}{n}}\in N^+(\HH)$ for some $n\in \N$, then $g$ must be cyclic. This suggests to consider logarithms.
\begin{thm}\label{thm:log} Let $\mu$ be a nonnegative finite Borel measure on $\T$, and let  $g\in D(\mu)$ be outer.

 If   $\log g\in N^+(D(\mu))$, then $g$ is cyclic in $D(\mu)$.
\end{thm}
Note that the theorem becomes false, if the requirement that $g$ be an outer function is removed. Indeed, if  $\mu=0$, then $D(\mu)=H^2$, and if $g$ is a singular inner function, then $\log g \in N^+(H^2)$, but $g$ is not cyclic in $H^2$.

We will prove the Theorem in Section \ref{sec:MainThm}. Thus,  inclusion (\ref{equ:N^+}) implies that for outer functions a sufficient condition for cyclicity is given by $\log g\in D(\mu)$. That condition was known,  see \cite{AlexandruHab}. However, for bounded outer functions the condition can be substantially improved.

Let $\C_+=\{z\in \C: \mathrm{Re} z>0\}$ and note that $\log(1+z)$ has positive real part on $\C_+$. Thus, for  $n\in \N_0$ we can inductively  define analytic functions on $\C_+$ by
\begin{align*} G_0(z)&=z,\\
G_{n+1}(z)&=\log(1+G_n(z)) \ \ \text{(principal branches)}.
\end{align*}
For radially weighted Besov spaces $\HH$ in the unit ball of $\C^d$ it was shown in \cite{APRSS2} that whenever $f$ has positive real part and  $G_n\circ f\in N^+(\HH)$, then  $f\in N^+(\HH)$. The same holds true for $\HH=D(\mu)$ and it can be applied with $f=\log \frac{1}{g}$ whenever $\|g\|_\infty \le 1$.
\begin{cor} \label{cor:N+Dmu}Let $g\in D(\mu)\cap H^\infty $ be an outer function such that $\|g\|_\infty \le 1$.

If there is $n\in \N$ such that $G_n\circ (\log\frac{1}{g})\in N^+(D(\mu))$, then $g$ is cyclic in $D(\mu)$.
\end{cor}

We will prove this Corollary in Section \ref{sec:iterations}. Thus, now the inclusion (\ref{equ:N^+}) implies the following Corollary.

\begin{cor}\label{cor:Dmu} Let $g\in D(\mu)\cap H^\infty $ be an outer function such that $\|g\|_\infty \le 1$.

If there is $n\in \N$ such that $G_n\circ (\log\frac{1}{g})\in D(\mu)$, then $g$ is cyclic in $D(\mu)$.
\end{cor}
We make several remarks. First, if $g$ satisfies the hypothesis of the Corollary, then $1+\log \frac{1}{g}$ has positive real part and hence is contained in $H^p$ for all $0<p<1$, see \cite{Garnett}, Chapter III, Theorem 2.4. Furthermore, for all $\alpha >0$ we have $$|G_1(\log \frac{1}{g(z)})|\le \log |1+\log \frac{1}{g(z)}| +\pi/2\le \frac{1}{\alpha}|1+\log \frac{1}{g(z)}|^{\alpha} + \pi/2.$$ Hence $G_1 \circ (\log\frac{1}{g})\in H^p$ for all $p<\infty$. But then $G_n\circ (\log\frac{1}{g})\in H^p$ for all $n\in \N$ and $p<\infty$. In particular, $G_n\circ (\log\frac{1}{g})\in H^2$.

Secondly, since $|G'_{n+1}|= \frac{|G'_n|}{|1+G_n|} \le |G'_n|$ one checks that the hypothesis of Corollary \ref{cor:Dmu} is more easily satisfied the larger $n$ is. For example, if $f$ is univalent and maps $\D$ to $\D\setminus \{0\}$, then $G_n\circ \log \frac{1}{f}$ is univalent, hence it will be in the Dirichlet space $D$, if and only if its range has finite area.
That  will definitely be the case if $G_n(\C_+)$ has finite area. The change of variables $x+iy=w=G_1(z)$ shows that $$\int_{\C_+}|G_2'(z)|^2 dA(z)=\int_0^\infty \int_{-\pi/2}^{\pi/2} \frac{1}{(1+x)^2+y^2}dxdy<\infty.$$ Hence,  Corollary \ref{cor:Dmu} recovers the known result that  any univalent function that maps the unit disc into $\D\setminus \{0\}$ must be cyclic in the Dirichlet space $D$, see \cite{RiSuJOT}.

Since the hypothesis of Corollary \ref{cor:Dmu} is more easily satisfied for larger $n$ one may wonder whether the same is true for  the hypothesis of Corollary \ref{cor:N+Dmu}. In Sections 4 and 5 we will show that for functions that satisfy a mild integrability condition the above conditions are equivalent.

\begin{thm}\label{thm:equivalences}
Let $g\in D(\mu) \cap H^\infty(\D)$ be an outer function with $\|g\|_\infty\le 1$ and such that either $g$ has bounded argument or $$\int_{\D}|g'(z)|^2 |G_n(\frac{1}{1-|z|^2} )|^2U_\mu(z)dA(z) <\infty$$ for some $n\in \N$.

 Then the following are equivalent
\begin{enumerate}
\item $g $ is cyclic in $D(\mu)$,
\item $ \log g\in N^+(D(\mu))$,
\item $\log(1+\log \frac{1}{g}) \in N^+(D(\mu))$,
\item there is $k \in \N$ such that $G_k\circ \log \frac{1}{g} \in N^+(D(\mu))$,
\item for all $k \in \N$ we have $G_k\circ \log \frac{1}{g} \in N^+(D(\mu))$.
\end{enumerate}
\end{thm}
 The Brown-Shields conjecture asks, whether an outer function $g$ is cyclic in the Dirichlet space $D$, whenever the radial zero set $Z(g)$ has logarithmic capacity 0, $Z(g)=\{e^{it}: \lim_{r\to 1^-}|g(re^{it})|=0\}$, see \cite{BrownShields}.

It is unclear that this conjecture can be resolved directly by use of Corollary \ref{cor:Dmu}. But if $g\in D$ is outer with $\|g\|_\infty\le 1$ and such that $Z(g)$ has logarithmic capacity 0, then one might try to find a function $f\in D$ with $Z(f)=Z(g)$ such that it is known that $f$ is cyclic, and such  that $f (G_n\circ (\log\frac{1}{g}))\in D$. Then $G_n\circ (\log\frac{1}{g})\in N^+(D)$ and $g$ would be cyclic by Corollary \ref{cor:N+Dmu}.

We showed that membership in the Pick-Smirnov class of the iterated logarithms $G_n(\log \frac{1}{g})$ determines cyclicity of $g$ for many functions $g\in D(\mu)$. It is thus interesting that $G_n(\log \frac{1}{g})$ is in the Pick-Nevanlinna class $N(D(\mu))$, whenever $g\in D(\mu)$ has no zeros in $\D$. Here
$$N(D(\mu))=\{\frac{\varphi}{\psi}: \varphi, \psi\in \mathrm{Mult}(D(\mu)), \psi(z)\ne 0 \text{ for all }z\in \D\}.$$
\begin{thm} \label{thm:Nevanlinna} If $g\in D(\mu)$ has no zeros in $\D$, then  $\log g \in N(D(\mu))$. Furthermore, if additionally $\|g\|_\infty \le 1$, then $G_n(\log \frac{1}{g}) \in N(D(\mu))$ for all $n\in \N$.
\end{thm}
The short proof will be given in Section \ref{sec:NevanlinnaClass}.

We would like to thank the referee for a careful reading of our original manuscript. Their insightful remarks inspired us to change the order of the  presentation of some theorems and their proofs. The changes  should make the paper more easily accessible.
We also thank Michael Hartz whose observation simplified our original proof of Theorem \ref{thm:log}. In fact, it made us realize that the results of this paper hold in the generality of all $D(\mu)$ spaces with $\mu$ supported in $\overline{\D}$, rather than only for $\mu$ supported in ${\partial \D}$.

\section{The proof of Theorem \ref{thm:log}} \label{sec:MainThm}

The only technical part of the proof of Theorem \ref{thm:log} is contained in the following Theorem.

\begin{thm}\label{thm:ghalpha}    If $h\in D(\mu)$ is an outer function such that $h^2\in D(\mu)$, then $h\in [h^2]$.
 \end{thm}
If $\mu$ is supported in $\partial \D$, then this a special case of Theorem 4.3 of \cite{RiSuJOT}. We will use the same approach as in \cite{RiSuJOT}, the key estimates will follow from results of \cite{AlexandruHab}.

We start by recalling the definition of the local Dirichlet integral. Let $f\in H^2$ and $z\in \overline{\D}$. If $z\in \D$ or if $z\in \partial \D$ and $f$ has nontangential limit $f(z)$ at $z$, then define
\begin{align}\label{equ:localDiri} D_z(f)=\int_{|w|=1} \left|\frac{f(w)-f(z)}{w-z}\right|^2 \frac{|dw|}{2\pi}.\end{align} Otherwise, set   $D_z(f)=\infty$.

\begin{lem}\label{lem:localDiriPotential} (=Theorem IV.1.9 of \cite{AlexandruHab}) If $f\in H^2$ and if $\mu$ is a nonnegative finite Borel measure in $\overline{\D}$, then
$$\int_{|z|\le 1} D_z(f) d\mu(z)= \int_{|w|<1} |f'(w)|^2 U_\mu(w) \frac{dA(w)}{\pi}.$$
\end{lem}
If $f$ and $g$ are outer functions in $H^2$, then we use $f\wedge g$ to denote the outer function such that $$|f\wedge g(e^{it})|= \min\{|f(e^{it})|, |g(e^{it})|\} \ a.e. $$ Similarly we define $f \vee g$ to be the outer function such that $$|f\vee g(e^{it})|= \max\{|f(e^{it})|, |g(e^{it})|\} \ a.e. $$

\begin{lem} \label{lem:wedge_vee} If $f\in H^2$ is outer and $z\in \overline{\D}$, then $D_z(f\wedge 1)\le D_z(f)$ and $D_z(f\vee 1)\le D_z(f)$.
 \end{lem}
If $|z|=1$, then this is a special case of Lemma 2.2 of \cite{RiSuJOT}. If $|z|<1$, then this is Proposition IV.3.3 (ii) and (iii) of \cite{AlexandruHab}. It can also be found in \cite{AlexandruProc1992}.

\begin{lem} \label{lem:fandfsquared} If $f\in H^2$ is outer and $z\in \overline{\D}$, then $D_z(f\wedge f^2) \le 10 D_z(f)$.
\end{lem}
\begin{proof} If $z\in \partial \D$, then this is Theorem 7.5.4 of \cite{ElFallahKellayMashreghiRansfordPrimer} with a constant 4 instead of 10. We will now verify the case $|z|<1$.

Write $F= f\vee 1$, then $|f/F(w)|\le 1$ for all $w\in \D$ and $f\wedge f^2= f^2/F$. Then for $w\in \D$
$$|(f\wedge f^2)'(w)|^2=\left|2\frac{f(w)}{F(w)}f'(w)-  \frac{f(w)^2}{F(w)^2} F'(w)\right|^2 \le 8|f'(w)|^2+ 2|F'(w)|^2.$$
By Lemma \ref{lem:localDiriPotential} with $\mu$ the unit point mass at $z$ this implies that
$$D_z(f\wedge f^2)\le 8 D_z(f) + 2 D_z(F) \le 10 D_z(f),$$ where the last inequality followed from Lemma \ref{lem:wedge_vee}.
\end{proof}

\begin{proof}[Proof of Theorem \ref{thm:ghalpha}] Let $h\in D(\mu)$ be an outer function such that $h^2\in D(\mu)$.
As in the proof of Lemma 9.1.5 of \cite{ElFallahKellayMashreghiRansfordPrimer} we apply the previous Lemma to the function $f=nh$ and divide by $n^2$ to obtain \begin{align}\label{equ:Dz(nh^2)}D_z(h\wedge n h^2)\le 10 D_z(h)\end{align} for each $n\in \N$ and each outer function $h\in H^2$. This implies that $h\wedge n h^2 \to h$ weakly in $D(\mu)$. Since $|h\wedge nh^2|\le n|h^2|$ we conclude from Corollary IV.4.4 of \cite{AlexandruHab} that $h\wedge nh^2\in [h]$ for each $n\in \N$, and hence $h\in [h^2]$.
\end{proof}

\

\begin{proof}[Proof of Theorem \ref{thm:log}] We know from equation (\ref{equ:Smirnovratio}) that $f=u/(1-v)$ for contractive multipliers $u,v$. By Lemma 3.2 of \cite{APRSS2} and the hypothesis it follows that $\log u\in N^+(D(\mu))$ and  we know that $u$ is cyclic if and only if $f$ is cyclic (see Lemma 6.3 of \cite{APRSS1}). Thus we will now assume that $f=u$ is a contractive multiplier.

The hypothesis implies that $f=e^{\varphi/\psi}$ for multipliers $\varphi, \psi$, where $\psi$ is cyclic.
For $n\in \N$ consider $g_n=\psi^2 f^{1/n}$. Since $\psi$ is cyclic, it must be outer, and that implies that each $g_n$ is outer.

Then $g_{2n}^2=\psi^2 g_n\in [g_n]$. By Theorem \ref{thm:ghalpha} we have $g_{2n}\in [g_{2n}^2]$. Thus, we obtain $g_{2n}\in [g_n]$ for each $n$, and that implies $g_{2^n}\in [g_1]\subseteq [f]$.

We will show that $g_{n}\to \psi^2$ in $D(\mu)$ as $n \to \infty$. Then it follows that $\psi^2 \in [f]$ and the cyclicity of $\psi$ implies $1\in [f]$, i.e. $f$ is cyclic.

We know $\|f\|_\infty \le 1$, and we fix $c>0$ such that
$\|\varphi\|_\infty, \|\psi\|_\infty \le c$. Then
$g_n'= 2\psi\psi'f^{1/n} +\frac{1}{n} f^{1/n}(\varphi'\psi-\psi'\varphi)$ and hence $$|g'_n-(\psi^2)'|\le 2c|\psi'||f^{1/n}-1| +\frac{c}{n}(|\varphi'|+|\psi'|).$$ Thus  $g_n \in D(\mu)$ for each $n\in \N$, and  by the Dominated Convergence Theorem

\begin{align*}&\int_\D|g_n'-(\psi^2)'|^2U_\mu dA\to 0 \ \ \text{ as } n \to \infty.
\end{align*}
It is clear that $\psi^2 f^{1/n}\to \psi^2$ pointwise in $\D$ as $n\to \infty$, hence $g_n\to \psi^2$ in the $D(\mu)$-norm.
\end{proof}
\section{Iterations and the proof of Corollary \ref{cor:N+Dmu}}\label{sec:iterations}
In the following lemma   $\log$ denotes the principal branch.
\begin{lem} \label{lem:RePositiv} Let $F\in \Hol(\D)$ such that $\mathrm{Re} F(z)\ge 0$.

If  $\log(1+ F)\in N^+(D(\mu))$, then $F \in N^+(D(\mu))$.
\end{lem}
For radially weighted Besov spaces this was proved in \cite{APRSS2}, Lemma 5.2 (a). The same proof works for $D(\mu)$. Since the $D(\mu)$-spaces are defined with just one derivative, the proof is very short and we repeat it here.

\begin{proof}  Suppose that $\log(1+F) \in N^+(D(\mu))$, then $\log(1+F)= \varphi/\psi$ for some multipliers $\varphi, \psi$, where $\psi$ is cyclic. Consider the function $h= \psi^2 e^{-\varphi/\psi}= \frac{\psi^2}{1+F}$. Then $h' = \frac{2\psi \psi'-\varphi'\psi+\varphi\psi'}{1+F}$ and hence $h \in D(\mu)$ since $|1+F|\ge 1$.

As in the proof of Theorem \ref{thm:log} we will now show that $h$ is cyclic in $D(\mu)$. For $n\in \N$ set $g_n=  \psi^2 e^{-\frac{1}{n} \varphi/\psi}$. Then one checks that $g_n \in D(\mu)$ and $g_n \to \psi^2$  as $n\to \infty$. Since $\frac{1}{1+F}=e^{-\varphi/\psi}$ is an outer function, we can apply Theorem \ref{thm:ghalpha} to conclude that $g_{2n}\in [g_{2n}^2]\subseteq [g_n]$ for each $n$. Thus, $g_{2^n} \in [h]$ for all $n\in \N$. Hence $\psi^2\in [h]$ and the cyclicity of $\psi^2$ implies that $h$ must be cyclic in $D(\mu)$.

 Thus, $h=\frac{u}{v}$ for two cyclic multipliers $u$ and $v$. It follows that $F=\frac{\psi^2 v-u}{u}\in N^+(D(\mu))$.
\end{proof}

It is now clear that a repeated application of this Lemma together with Theorem \ref{thm:log} imply Corollary \ref{cor:N+Dmu}.

\section{A sufficient condition for the converse of Corollary \ref{cor:N+Dmu}}
We start with  a Lemma.
\begin{lem}\label{lem:G_k_k>n} Let $j,n\in \N$. If $g\in D(\mu)\cap H^\infty$ such that $g^j (G_n \circ \log \frac{1}{g})\in D(\mu)$, then $g^j (G_k \circ \log \frac{1}{g})\in D(\mu)$ for all $k \ge n$.
\end{lem}
\begin{proof} We show the statement for $k=n+1$, then the general case follows by induction.
First we note that $|G_{n+1}(z)|\le \log (1+|G_n(z)|) +\pi/2 \le |G_n(z)|+\pi/2$. Hence the hypothesis easily shows that $g^j (G_{n+1} \circ \log \frac{1}{g})\in H^2$. For later reference we also observe that $g\in D(\mu)\cap H^\infty$ implies that  $g^{j-1}g'\in L^2(U_\mu dA)$.

Furthermore, one checks that $|G_k'(w)|\le 1$ for all $w\in \C_+$ and all $k\in \N$. Hence $|g^j(z) (G_k(\log \frac{1}{g(z)}))'| = |g(z)|^{j-1} |g'(z)||G_k'(\log \frac{1}{g(z)})|\in L^2(U_\mu dA)$ for $k=n$ and $k=n+1$. This implies that $g^j (G_k\circ \log\frac{1}{g})\in D(\mu)$, if and only if $g^{j-1}g' (G_k\circ \log\frac{1}{g})\in L^2(U_\mu dA)$. But by the above estimate we have
$$|g^{j-1}(z)g'(z) (G_{n+1}\circ \log\frac{1}{g(z)})|\le |g^{j-1}(z)g'(z) (G_{n}\circ \log\frac{1}{g(z)})| + \frac{\pi}{2} |g^{j-1}(z)g'(z)|.$$ The Lemma follows.
\end{proof}
In the next section we will prove that the following Theorem implies Theorem \ref{thm:equivalences}.
\begin{thm}\label{thm:equivalences1}
Let $g\in D(\mu) \cap H^\infty(\D)$ be an outer function with $\|g\|_\infty\le 1$ and such that $g^j (G_n\circ\log \frac{1}{g})\in D(\mu)$ for some $j,n\in \N$.

 Then the following are equivalent
\begin{enumerate}
\item $g $ is cyclic in $D(\mu)$,
\item $ \log g\in N^+(D(\mu))$,
\item $\log(1+\log \frac{1}{g}) \in N^+(D(\mu))$,
\item there is $k \in \N$ such that $G_k\circ \log \frac{1}{g} \in N^+(D(\mu))$,
\item for all $k \in \N$ we have $G_k\circ \log \frac{1}{g} \in N^+(D(\mu))$.
\end{enumerate}
\end{thm}
\begin{proof} It is clear that $(5)\Rightarrow (4)$. The implications $(4)\Rightarrow (3)\Rightarrow (2)$  follow from Lemma \ref{lem:RePositiv}, and $(2)\Rightarrow (1)$ is Theorem \ref{thm:log}.

We next show $(1) \Rightarrow (4)$ with $k=n$. Then Lemma \ref{lem:G_k_k>n} will imply  that $G_k\circ \log \frac{1}{g} \in N^+(D(\mu))$ for every $k \ge n$.

We have $g=u/v$ and $g^j (G_n \circ \log \frac{1}{g})= \varphi/\psi$
for multipliers $u, v, \varphi$, and $\psi$, where $v$ and $ \psi$ are cyclic. But then $[g]=[u]$ by Lemma 6.3 of \cite{APRSS1}. Thus, the cyclicity of $g$ implies that $u$ is cyclic, and this implies that
$G_n \circ \log \frac{1}{g}= \frac{\varphi v^j}{\psi u^j}\in N^+(D(\mu))$, since $\psi u^j$ is cyclic. Here we used the easy observation that a product of finitely many multipliers is cyclic, if and only if each factor is cyclic.

Finally, by what we have just shown in order to establish $(1)\Rightarrow (5)$ we only have to note that $G_n\circ \log \frac{1}{g} \in N^+(D(\mu))$  implies
 $G_k\circ \log \frac{1}{g} \in N^+(D(\mu))$ for all $k\in \N$ with $k\le n$ by Lemma \ref{lem:RePositiv}.\end{proof}

\section{Establishing the hypothesis of Theorem \ref{thm:equivalences1}} \label{sec:thmgrowth}
In \cite{APRSS2}  Theorem \ref{thm:equivalences} was proved for radial weights under the hypothesis that $g$ has bounded argument. In fact, if $g$ has bounded argument and $\|g\|_\infty \le 1$, then $g (G_n \circ \log \frac{1}{g}) $ is bounded for each $n$, and  it easily follows that $g^2 (G_n \circ \log \frac{1}{g})\in D(\mu)$ for each $n\in \N$. Thus, in this case Theorem \ref{thm:equivalences} follows from Theorem \ref{thm:equivalences1} with $j=2$.

The remaining case of Theorem \ref{thm:equivalences} follows from Theorem \ref{thm:equivalences1} with $j=1$ as the following Theorem shows.

\begin{thm}\label{thm:growthcond} Let $g\in D(\mu) \cap H^\infty(\D)$ be a function such that $g(z)\ne 0$ for all $z\in \D$ and $\|g\|_\infty\le 1$, and let $n\in \N$.

If $\int_{\D}|g'(z)|^2 |G_n(\frac{1}{1-|z|^2} )|^2U_\mu(z)dA(z) <\infty$, then $g (G_n \circ \log \frac{1}{g})\in D(\mu)$.
\end{thm}

The proof requires a little bit of set-up and a lemma.
One checks that
\begin{align}\label{equ:RGn}|G_n'(z)| \le 1, \ \mathrm{Re }z>0, \ n\in \N_0.\end{align}

Next, for $n\in \N_0$ define functions $F_n$ on $[0,1)$ by $F_n(r)=G_n(\frac{1}{1-r})$. Then each $F_n$ is continuous and increases to $\infty$ as $r\to 1^-$, and the functions satisfy $F_n(0)>0$ and $F_{n+1}(r)\le F_n(r)$ for all $n\in \N_0$ and $r\in [0,1)$.  If $M_0>0$ is given, then for each $n\ge 1$  we have
 $M_n= \sup_{x\in [F_n(0),\infty)}\frac{\log(1+\pi/2+ M_{n-1} x)}{\log (1+x)}<\infty$ and hence  \begin{align}\label{equ:F_est}\log(1+\pi/2 + M_{n-1}F_{n-1}(r)) \le M_n F_n(r)\end{align} for all $r\in [0,1)$.

\begin{lem}\label{lem:nbounded}  If $f\in \Hol(\D)$ with $\mathrm{Re} f(z)> 0$,
then  for each $n\in \N_0$ we have $$|G_n(f(z))|\le \frac{\pi}{2} + M_n F_n(|z|^2)$$ for all $z\in \D$. Here $\{M_n\}$ is defined as above with $M_0=5|f(0)|$. \end{lem}
\begin{proof} We will prove the Lemma by induction on $n$ with the sequence $M_n$ as above defined by $M_0=5| f(0)|$.
Since $f$ has positive real part, there is a positive measure $\nu$ on $[0,2\pi]$ such that
$$f(z)= \int_0^{2\pi} \frac{e^{it}+z}{e^{it}-z}d\nu(t) + i \mathrm{Im}f(0).$$
Then $$|f(z)|\le \frac{4 \mathrm{Re } f(0)}{1-|z|^2} + |\mathrm{Im}f(0)| \le \frac{5 |f(0)|}{1-|z|^2}.$$

Thus, we obtain
$$|G_0(f(z))|=|f(z)|\le M_0 F_0(|z|^2) \le \frac{\pi}{2} + M_0 F_0(|z|^2).$$

Now  assume that the Lemma holds for $n\ge 0$. Then
\begin{align*}|G_{n+1}(f(z))|&= |\log ( 1 + G_n(f(z)))|\\
&\le \frac{\pi}{2} + \log (1+|G_n(f(z))|)\\
&\le  \frac{\pi}{2} + \log(1+\pi/2+M_nF_n(|z|^2)) \ \text{ by induction hypothesis}\\
&\le \frac{\pi}{2} + M_{n+1}F_{n+1}(|z|^2) \ \text{ by }(\ref{equ:F_est}).\end{align*}
The Lemma follows.
\end{proof}

\begin{proof}[Proof of Theorem \ref{thm:growthcond}] Let $g\in D(\mu) \cap H^\infty(\D)$  such that $g(z)\ne 0$ for all $z\in \D$, $\|g\|_\infty\le 1$,
and \begin{align}\label{equ:hyp}\int_{\D}|g'(z)|^2 |F_n(|z|^2 )|^2U_\mu(z)dA(z) <\infty\end{align} for some $n\in \N$.
We have to show that  $h=g G_n(\log \frac{1}{g})\in D(\mu)$.

 By inequality  (\ref{equ:RGn}) we have $$|h'(z)|\le  |g'(z)G_n(\log \frac{1}{g(z)})| + |g(z) G_n'(\log \frac{1}{g(z)})\frac{g'(z)}{g(z)}| \le |g'(z)|(|G_n(\log \frac{1}{g(z)})|+1), $$ hence an application of  Lemma \ref{lem:nbounded} with $f=\log \frac{1}{g}$ implies that $$|h'(z)| \le |g'(z)|( 1+\frac{\pi}{2} + M_n F_n(|z|^2)).$$ Thus, the Theorem follows from the hypothesis \ref{equ:hyp}
\end{proof}

\section{The proof of Theorem \ref{thm:Nevanlinna}}\label{sec:NevanlinnaClass}

\begin{proof}[Proof of Theorem \ref{thm:Nevanlinna}]  Since $D(\mu)$ is a Pick space we have $g= \frac{u}{1-v}$ for two contractive multipliers $u$ and $v$. By Lemma 3.2 of \cite{APRSS2} we have $\log({1-v})\in N^+(D(\mu))$, hence we may assume that $g$ is a multiplier, $g(z)\ne 0$ for all $z\in \D$, and $\|g\|_\infty \le 1$.

Thus, let $n\in \N \cup\{0\}$ and set $f= \frac{g}{1+G_n(\log \frac{1}{g})}$. We will show that $f\in D(\mu)$. Then $f=\frac{\varphi}{\psi}$ for multipliers $\varphi, \psi$, where $\psi $ is cyclic. Furthermore, it is clear that $\varphi(z)\ne 0$ for all $z\in \D$. Hence we will have $$G_n(\log \frac{1}{g})= \frac{g\psi}{\varphi}-1\in N(D(\mu)).$$

 We have
$$|f'(z)| =\left|\frac{g'(z)}{1+G_n(\log \frac{1}{g(z)})} +\frac{g'(z)}{(1+G_n(\log \frac{1}{g(z)}))^2}G_n'(\log \frac{1}{g(z)})\right|\le 2|g'(z)|, $$ where we used that $|1+G_n(w)|\ge 1$ and $|G_n'(w)|\le 1$ for all $w\in \C$ with $\mathrm{Re}w>0$. These inequalities can easily be verified by induction. Hence $f\in D(\mu)$.
\end{proof}

\bibliography{CyclicBib2}
\end{document}